\documentclass{amsart}
\usepackage{latexsym,amsxtra,amscd,ifthen}
\usepackage{amsfonts}
\usepackage{color,graphicx}
\usepackage{verbatim}
\usepackage{amsmath}
\usepackage{amsthm}
\usepackage{amssymb}
\usepackage{url}

\theoremstyle{plain}

\newtheorem{theorem}{Theorem}
\newtheorem{lemma}[theorem]{Lemma}
\newtheorem{proposition}[theorem]{Proposition}
\newtheorem{corollary}[theorem]{Corollary}

\numberwithin{theorem}{section}
\numberwithin{equation}{theorem}

\theoremstyle{definition}
\newtheorem{definition}[theorem]{Definition}
\newtheorem{example}[theorem]{Example}
\newtheorem{remark}[theorem]{Remark}
\newtheorem{question}[theorem]{Question}
\newtheorem*{question*}{Question}

\newcommand{\fm}{\mathfrak{m}}

\DeclareMathOperator{\End}{End}

\DeclareMathOperator{\Hom}{Hom}
\DeclareMathOperator{\Aut}{Aut}
\DeclareMathOperator{\gr}{gr}
\DeclareMathOperator{\tr}{tr}

\DeclareMathOperator{\GKdim}{GKdim}
\DeclareMathOperator{\MaxSpec}{MaxSpec}
\DeclareMathOperator{\Spec}{Spec}

\DeclareMathOperator{\LND}{LND}

\begin{document}

\title{Morita cancellation problem}

\author{D.-M. Lu, Q.-S. Wu and J.J. Zhang}

\address{Lu:
School of Mathematical Sciences, 
Zhejiang University, Hangzhou 310027, China}

\email{dmlu@zju.edu.cn}

\address{Wu: 
School of Mathematical Sciences, 
Fudan University, Shanghai 200433, China}

\email{qswu@fudan.edu.cn}

\address{Zhang: Department of Mathematics, Box 354350,
University of Washington, Seattle, Washington 98195, USA}

\email{zhang@math.washington.edu}

\begin{abstract}
We study a Morita-equivalent version of the Zariski cancellation problem.
\end{abstract}


%16P99  View Publications (1991-now) None of the above, but in this section
%16W99  View Publications (1991-now) None of the above, but in this section
% Rings and algebras with additional structure
%16W20 (1991-now) Automorphisms and endomorphisms
%16E65  View Publications  (2000-now) Homological
%conditions on rings (generalizations of regular,
%Gorenstein, Cohen-Macaulay rings, etc.)


%\thanks{ }

\maketitle

%\tableofcontents

% \setcounter{section}{-1}
\section*{Introduction}
\label{xxsec0}

An algebra $A$ is called {\it cancellative} if 
any algebra isomorphism $A[t] \cong B[t]$ of polynomial algebras 
for some algebra $B$ implies that $A$ is isomorphic to $B$.
% Here $A[t]$ denotes the polynomial extension of $A$.
The famous Zariski Cancellation Problem (abbreviated as ZCP) asks if 

\medskip

{\it the commutative polynomial ring $\Bbbk [x_1,\dots,x_n]$ over a field 
$\Bbbk$ is cancellative}

\medskip

\noindent 
for $n\geq 1$ \cite{Kr, BZ1, Gu3}. There is a 
long history of studying the cancellation property of affine commutative 
domains. For example, $\Bbbk[x_1]$ is cancellative by a result of 
Abhyankar-Eakin-Heinzer in 1972 \cite{AEH}, while $\Bbbk[x_1,x_2]$ is 
cancellative by a result of Fujita in 1979 \cite{Fu} and Miyanishi-Sugie 
in 1980 \cite{MS} in characteristic zero and by a result of Russell in 
1981 \cite{Ru} in positive characteristic. The ZCP for $n\geq 3$ has 
been open for many years. One remarkable achievement in this research 
area is a result of Gupta in 2014 \cite{Gu1, Gu2} which settled the 
ZCP negatively in positive characteristic for $n\geq 3$. The ZCP in 
characteristic zero remains open for $n\geq 3$.

The ZCP (especially in dimension two) is closely related to the 
Automorphism Problem, the Characterization Problem, the Linearization 
Problem, the Embedding Problem, and the Jacobian Conjecture, see 
\cite{Kr, EH, Gu3, BZ1} for history, partial results and references 
concerning the cancellation problem.

The ZCP for noncommutative algebras was introduced 
in \cite{BZ1} and further investigated in \cite{LWZ}. During the last 
few years, several researchers have been making significant contributions 
to the cancellation problem in the noncommutative setting and related topics, 
see papers \cite{BZ1, BZ2, BY, CPWZ1, CPWZ2, CYZ1, CYZ2, Ga, GKM, GWY, LY, 
LWZ, LMZ, NTY, Ta1, Ta2, WZ} and so on.

The first goal of this paper is to introduce a new cancellation property 
for noncommutative algebras. Let $\Bbbk$ be a base field and everything 
is over $\Bbbk$. For any algebra $A$, let $M(A)$ denote the category of 
right $A$-modules. 

\begin{definition}
\label{xxdef0.1} An algebra $A$ is called {\it Morita cancellative}
if the statement that 
$${\text{$M(A[t])$ is equivalent to $M(B[t])$ for another algebra $B$}}$$
implies that
$${\text{$M(A)$ is equivalent to $M(B)$.}}$$
\end{definition}

This Morita version of the cancellation property is one of the natural 
generalizations of the original Zariski cancellation property when we study 
noncommutative algebras. Another generalization involves the derived category 
of modules. Let $D(A)$ denote the derived category 
of right $A$-modules for an algebra $A$. 

\begin{definition}
\label{xxdef0.2} An algebra $A$ is called {\it derived cancellative}
if the statement that 
$${\text{$D(A[t])$ is triangulated equivalent to $D(B[t])$  for an algebra $B$}}$$
implies that
$${\text{$D(A)$ is triangulated equivalent to $D(B)$.}}$$
\end{definition}

We will show that [Theorem \ref{xxthm0.7}] if $Z$ is a commutative domain, 
then 
$${\text{$Z$ is Morita cancellative if and only if $Z$ is cancellative}}$$
and
$${\text{$Z$ is derived cancellative if and only if  $Z$ is cancellative.}}$$ 
In general, when $A$ is noncommutative, it is not clear to us what are the 
relationships between these three different versions of cancellation property. 
Lemma \ref{xxlem1.4} (together with Example \ref{xxex1.5}) provides 
noncommutative algebras 
that are neither cancellative, nor Morita cancellative, nor derived 
cancellative. We will introduce some general methods to handle the Morita 
cancellation problems for noncommutative algebras. 

The second aim of the paper is to show several classes of algebras 
are Morita (or derived) cancellative. First we generalize a result 
of \cite[Theorem 0.2]{LWZ}.

\begin{theorem}
\label{xxthm0.3}
Suppose $A$ is strongly Hopfian {\rm{(}}Definition \ref{xxdef3.2}{\rm{)}}
and the center of $A$ is artinian. Then $A$ is Morita cancellative.
\end{theorem}

Note that left (or right) noetherian algebras and locally finite
${\mathbb N}$-graded algebras are strongly Hopfian [Example \ref{xxex3.5}].
So Theorem \ref{xxthm0.3} covers a large class of algebras. The following
are consequences of the above theorem, see also 
\cite[Corollary 0.3 and Theorem 0.4]{LWZ} for comparison.

\begin{theorem}
\label{xxthm0.4} Let $A$ be a left {\rm{(}}or right{\rm{)}}
noetherian algebra such that its center 
is artinian. Then $A$ is Morita cancellative. As a consequence, 
every finite dimensional algebra over a base field $\Bbbk$ is
Morita cancellative.
\end{theorem}

For non-noetherian algebras we have the following.

\begin{theorem}
\label{xxthm0.5}
For every finite quiver $Q$, the path algebra $\Bbbk Q$ is
Morita cancellative.
\end{theorem}

Recall from \cite[Theorem 0.5]{BZ1} that, if $A$ is an affine domain of 
GK-dimension two over an algebraically closed field of characteristic
zero and if $A$ is not commutative, then $A$ is cancellative. It is well-known
that, in contrast, noncommutative affine prime (non-domain) algebras of 
GK-dimension two need not be cancellative \cite[Example 1.3(5)]{LWZ}
and that commutative affine domains of GK-dimension two need not be 
cancellative by examples of Hochster \cite{Ho} and Danielewski \cite{Da},
see Example \ref{xxex1.5}(1,2). For GK-dimension one, a classical result 
of Abhyankar-Eakin-Heinzer \cite[Theorem 3.3]{AEH} says that every 
affine commutative domain of GK-dimension one is cancellative. Recently, 
it was proved that every affine prime $\Bbbk$-algebra of GK-dimension 
one is cancellative. Next we are adding another result in low GK-dimension.

\begin{theorem}
\label{xxthm0.6}
Let $\Bbbk$ be algebraically closed. Then every affine prime 
$\Bbbk$-algebra of GK-dimension one is Morita 
cancellative.
\end{theorem}

We are mainly dealing with the Morita cancellation property in this paper, 
but occasionally, we have some results concerning the derived 
cancellation property such as the next result.

\begin{theorem}[Corollary \ref{xxcor7.2}]
\label{xxthm0.7}
Let $Z$ be a commutative domain. Then $Z$ is cancellative if and only
if $Z$ is Morita cancellative, if and only if $Z$ is derived cancellative.
\end{theorem}

A question in \cite[Question 5.4(3)]{LWZ} asks if the Sklyanin algebras 
are cancellative. We partially answer this question.

\begin{corollary}[Example \ref{xxex5.10}(2)]
\label{xxcor0.8}
Let $A$ be a non-PI Sklyanin algebra of global dimension three.
Then $A$ is both cancellative and Morita cancellative.
\end{corollary}

The paper is organized as follows. Section 1 contains definitions, known
examples and preliminaries. In Sections 2 and 3, we introduce the Morita 
version of the retractable and detectable properties. In Section 4, we 
prove Theorems \ref{xxthm0.3} and \ref{xxthm0.4}. Theorems \ref{xxthm0.6} 
and \ref{xxthm0.7} are proven in Section 6 and Section 5 respectively. 
The derived cancellation property is briefly studied in Section 7. 
Section 7 also contains some comments, remarks and examples.

\section{Definitions and Preliminaries}
\label{xxsec1}

Some definitions and examples are copied from \cite{BZ1, LWZ}. First we 
recall a classical definition. Let $A[t]$ (or $A[s]$) be the polynomial 
algebra over $A$ by adding one central indeterminate.

\begin{definition}
\label{xxdef1.1}
Let $A$ be an algebra.
\begin{enumerate}
\item[(1)]
We call $A$ {\it cancellative} if any algebra isomorphism 
$A[t] \cong B[s]$ for some algebra $B$ implies that $A\cong B$.
\item[(2)]
We call $A$ {\it strongly cancellative} if, for each $n \geq 1$, 
any algebra isomorphism
$$A[t_1, \cdots, t_n]\cong B[s_1,\cdots,s_n]$$
 for some algebra $B$ implies that $A \cong B$.
\end{enumerate}
\end{definition}

The following are two new cancellation properties which we will study 
in the present paper.

\begin{definition}
\label{xxdef1.2}
Let $A$ be an algebra.
\begin{enumerate}
\item[(1)]
We call $A$ {\it m-cancellative} if any equivalence of abelian categories
$M(A[t]) \cong M(B[s])$  for some algebra $B$ implies that $M(A)\cong M(B)$.
\item[(2)]
We call $A$ {\it strongly m-cancellative} if, for each $n \geq 1$, 
any equivalence of abelian categories
$$M(A[t_1, \cdots, t_n])\cong M(B[s_1,\cdots,s_n])$$
for some algebra $B$ implies that $M(A) \cong M(B)$. 
\end{enumerate}
The letter $m$ here stands for the word ``Morita''.
\end{definition}

\begin{definition}
\label{xxdef1.3}
Let $A$ be an algebra.
\begin{enumerate}
\item[(1)]
We call $A$ {\it d-cancellative} if any equivalence of triangulated categories
$$D(A[t]) \cong D(B[s])$$ for some algebra $B$ implies that $D(A)\cong D(B)$.
\item[(2)]
We call $A$ {\it strongly d-cancellative} if, for each $n \geq 1$, 
any equivalence of triangulated categories
$$D(A[t_1, \cdots, t_n])\cong D(B[s_1,\cdots,s_n])$$
for some algebra $B$ implies that $D(A) \cong D(B)$. 
\end{enumerate}
The letter $d$ here stands for the word ``derived''.
\end{definition}

Let $A[\underline{t}]$ denote the polynomial algebra $A[t_1,\cdots,t_n]$
and $A[\underline{s}]$ denote  the polynomial algebra $A[s_1,\cdots,s_n]$
for an integer $n$ (that is not specified) 
when no confusion occurs.

\begin{lemma}
\label{xxlem1.4} Let $A$ be a commutative algebra that is not 
{\rm{(}}strongly{\rm{)}} cancellative. Let $B$ be an algebra
with center $Z(B)=\Bbbk$. Then $A\otimes B$ is neither
{\rm{(}}strongly{\rm{)}}  cancellative, nor 
{\rm{(}}strongly{\rm{)}}  m-cancellative, nor 
{\rm{(}}strongly{\rm{)}}  d-cancellative.
\end{lemma}

\begin{proof} %We show only the ``strongly'' version for 
%d-cancellative property. The proof for the other cancellative 
%properties is the same. 
%
Since $A$ is not (strongly) cancellative, there is a commutative
algebra $C$ such that $A$ is not isomorphic to $C$, but $A[t_1,\cdots,
t_n]\cong C[s_1,\cdots, s_n]$ for $n=1$ (or some $n\geq 1$). Then
$A\otimes B [\underline{t}]\cong C\otimes B[ \underline{s}]$.
As a consequence, we obtain that 
$$M(A\otimes B[\underline{t}])\cong M(C\otimes B[\underline{s}]) \textrm{ and }
D(A\otimes B[\underline{t}])\cong D(C\otimes B[\underline{s}]).$$
Since the center $Z(A\otimes B)=A$ which is not isomorphic to
$Z(C\otimes B)=C$, we obtain that $M(A\otimes B) \ncong M(C\otimes B)$  
and that $D(A\otimes B) \ncong D(C\otimes B)$. Therefore the assertions 
follow.
\end{proof}

Next we give some precise examples of non-cancellative 
commutative algebras. The above lemma gives an easy way of 
producing non-cancellative noncommutative algebras.

\begin{example}
\label{xxex1.5}
\begin{enumerate}
\item[(1)]
Let $\Bbbk$ be the field of real numbers ${\mathbb R}$.
Hochster showed that $\Bbbk[x,y,z]/(x^2+y^2+z^2-1)$ 
is not cancellative \cite{Ho}.
\item[(2)]
The following example is due to Danielewski \cite{Da}. 
Let $n\geq 1$ and let $B_n$ be the coordinate ring of 
the surface $x^n y=z^2-1$ over $\Bbbk:={\mathbb C}$. 
Then $B_i\not\cong B_j$ if $i\neq j$, but $B_i[t]\cong 
B_j[s]$ for all $i,j\geq 1$, see \cite{Fi, Wi} for more details.
Therefore, all the $B_n$'s are not cancellative.
\item[(3)]
Suppose ${\rm char}\;\Bbbk>0$. Gupta showed that 
$\Bbbk[x_1,\cdots, x_n]$ is not cancellative for every $n\geq 3$
\cite{Gu1,Gu2}.
\end{enumerate}
As a consequence of Lemma \ref{xxlem1.4} (by taking $B=\Bbbk$), 
the algebras above are neither m-cancellative nor d-cancellative. 
\end{example}

We also need to recall higher derivations and Makar-Limanov 
invariants.

\begin{definition}
\label{xxdef1.6} Let $A$ be an algebra.
\begin{enumerate}
\item[(1)] \cite{HS}
A {\it higher derivation} (or {\it Hasse-Schmidt derivation})
on $A$ is a sequence of $\Bbbk$-linear endomorphisms
$\partial:=\{\partial_i\}_{i=0}^{\infty}$ such that:
$$\partial_0 = id_A, \quad
\text{and} \quad  \partial_n(ab) =\sum_{i=0}^n
\partial_i(a)\partial_{n-i}(b)
$$
for all $a, b \in A$ and all $n\geq 0$. The collection of
all higher derivations on $A$ is denoted by ${\rm Der}^H(A)$.
\item[(2)]
A higher derivation is called {\it locally nilpotent} if
\begin{enumerate}
\item
given any $a\in A$ there exists $n\geq 1$ such that 
$\partial_i(a)=0$ for all $i\geq n$,
\item
the map
$$G_{\partial,t}:A[t]\to A[t] \qquad\qquad\qquad\qquad $$
defined by
$$ \qquad\qquad \qquad\qquad
a\mapsto \sum_{i=0}^{\infty} \partial_i(a)t^i 
{\text{ for all $a\in A$, and }} \; t\mapsto t
$$
is an algebra automorphism of $A[t]$.
\end{enumerate}
\item[(3)]
For any $\partial\in {\rm Der}^H(A)$, the kernel of $\partial$
is defined to be
$$\ker \partial =\bigcap_{i\geq 1} \ker \partial_i.$$
\item[(4)]
The set of locally nilpotent higher derivations is denoted by
$\LND^H(A)$. Given a nonzero element $d\in A$, let 
\begin{equation}
\label{E1.6.1}
\LND^H_{d}(A)=\{\partial \in \LND^H(A) \mid d \in \ker\partial\}.
\end{equation}
\end{enumerate}
\end{definition}

Note that (a) in part (2) of the above definition implies that 
the map $G_{\partial,t}$ defined in (b) is an algebra endomorphism.
It is not clear to us whether $G_{\partial,t}$ is automatically 
an automorphism. However, by \cite[Lemma 2.2(2)]{BZ1}, when 
$\partial$ is an iterative higher derivation, $G_{\partial,t}$
is automatically an automorphism.

It is easy to see that $1 \in \ker \partial$ for all higher 
derivations $\partial$. Hence
$\LND^H_{1}(A)=\LND^H(A)$. We generalize the original definition
of the Makar-Limanov invariant \cite{Mak}.

\begin{definition}
\label{xxdef1.7}
Let $A$ be an algebra and $d$ a nonzero element in $A$.
\begin{enumerate}
\item[(1)] 
The {\it Makar-Limanov$^H_d$ invariant} of $A$ is defined to be
\begin{equation}
\label{E1.7.1}\tag{E1.7.1}
ML^H_d(A) := \bigcap_{\delta\in {\rm LND}^H_d(A)} {\rm ker}(\delta).
\end{equation}
\item[(2)] 
We say that $A$ is \emph{$\LND^H_d$-rigid} if $ML^H_d(A)=A$.
\item[(3)]
$A$ is called \emph{strongly $\LND^H_d$-rigid} if 
$ML^H_d(A[t_1,\ldots,t_n])=A$, for all $n\geq 1$.
\item[(4)]
The {\it Makar-Limanov$^H_d$ center} of $A$ is defined to be
\begin{equation}
\notag
ML^H_{d,Z}(A) =ML^H_d(A) \cap Z(A).
\end{equation}
\item[(5)]
$A$ is called \emph{strongly $\LND^H_{d,Z}$-rigid} if 
$ML^H_{d,Z}(A[t_1,\ldots,t_n])=Z(A)$, for all
$n\geq 1$.
\end{enumerate}
\end{definition}

\section{Morita invariant properties and the ${\mathcal P}$-discriminant}
\label{xxsec2}
In this section we will recall some well-known facts about 
Morita equivalence. Two algebras $A$ and $B$ are 
Morita equivalent if their right module categories $M(A)$
and $M(B)$ are equivalent. We list some properties 
concerning Morita theory. 

\begin{lemma}
\cite[Ch. 6]{AF}
\label{xxlem2.1} 
Let $A$ and $B$ be two algebras that are Morita equivalent.
\begin{enumerate}
\item[(1)]
There is an $(A,B)$-bimodule $\Omega$ that is invertible, namely,
$\Omega\otimes_B \Omega^{\vee}\cong A$ and $\Omega^{\vee}\otimes_A
\Omega\cong B$ as bimodules, where $\Omega^{\vee}:=\Hom_{B}(\Omega_B, B_B)$.
\item[(2)]
The bimodule $\Omega$ induces naturally algebra isomorphisms 
$A \cong \End(\Omega_B)$ and $ B^{op} \cong \End(_A\Omega)$.
\item[(3)]
Further, $Z(A)\cong \Hom_{(A,B)}(\Omega, \Omega)\cong Z(B)$, which induces 
an isomorphism
\begin{equation}
\label{E2.1.1}\tag{E2.1.1}
\omega: Z(A)\to Z(B)
\end{equation}
such that, for each $x\in Z(A)$, the left multiplication 
of $x$ on $\Omega$ equals the right multiplication of $\omega(x)$
on $\Omega$.
\item[(4)]
By using $\omega$ to identify the center $Z=Z(A)$ of $A$ with 
the center  of $B$, then both $A$ and $B$ are central $Z$-algebras. 
In this case both $\Omega$ and $\Omega^{\vee}$ are central $Z$-modules. 
\item[(5)]
Let $\omega$ be given as in \eqref{E2.1.1}. Then, for any ideal $I$ of 
$Z(A)$, $A/IA$ and $B/\omega(I)B$ are Morita equivalent.
\item[(6)]
\cite[Ex.9, p.267]{AF}
Let $A,B,T$ be $K$-algebra for some commutative ring $K$.
Then $A\otimes_K T$ and $B\otimes_K T$ are Morita equivalent. 
\end{enumerate}
\end{lemma}

Morita equivalences have been studied extensively for decades. A ring 
theoretic property is called a {\it Morita invariant} if it is preserved 
by Morita equivalences. 

\begin{example}
\label{xxex2.2}
The following properties are Morita invariants.
\begin{enumerate}
\item[(1)]
being simple (respectively, semisimple);
\item[(2)]
being right (or left) noetherian, right (or left) artinian;
\item[(3)]
having global dimension $d$ (Krull dimension $d$, 
GK-dimension $d$, etc);
\item[(4)]
being a full matrix
algebra $M_n(\Bbbk)$ for some $n$, when 
$\Bbbk$ is algebraically closed; 
\item[(5)]
being an Azumaya algebra \cite[Theorem 4]{Sc};
\item[(6)]
being quasi-Frobenius;
\item[(7)]
being prime, semiprime, right (or left) primitive,  semiprimitive;
\item[(8)]
being semilocal, 
\item[(9)]
being primitive, but not simple;
\item[(10)]
being noetherian, but not artinian;
\item[(11)]
the center being $\Bbbk$;
\item[(12)]
being projective over its center.
\end{enumerate}
\end{example}

Let $R$ be a commutative algebra, $\Spec R$ be the prime spectrum of $R$ and 
$\MaxSpec(R):=\{\fm \mid \fm {\text{ is a maximal ideal of }} R\}$ be the 
maximal spectrum of $R$. For any $S \subseteq \Spec R$, $I(S)$ is the ideal 
of $R$ vanishing on $S$, namely,
$$I(S)=\bigcap_{{\mathfrak p}\in S} {\mathfrak p}.$$ 

For any algebra $A$, $A^{\times}$ denotes the 
set of invertible elements in $A$. 

A property ${\mathcal P}$ considered in the following
means a property defined on a class of algebras that 
is an invariant under algebra isomorphisms. 
\begin{definition}
\label{xxdef2.3} 
Let $A$ be an algebra, $Z=Z(A)$ be the center of $A$. 
Let ${\mathcal P}$ be a property defined for $\Bbbk$-algebras
{\rm{(}}not necessarily a Morita invariant{\rm{)}}. 
\begin{enumerate}
\item[(1)]
The {\it ${\mathcal P}$-locus} of $A$ is defined to be
$$L_{\mathcal P}(A):=\{ \fm \in \MaxSpec(Z)\mid A/\fm A {\text{ has 
property }} {\mathcal P}\}.$$
\item[(2)]
The {\it ${\mathcal P}$-discriminant set} of $A$ is defined to be
$$D_{\mathcal P}(A):=\MaxSpec(Z)\setminus L_{\mathcal P}(A).$$
\item[(3)]
The {\it ${\mathcal P}$-discriminant ideal} of $A$ is defined to be
$$I_{\mathcal P}(A):=I(D_{\mathcal P}(A))\subseteq Z.$$
\item[(4)]
If $I_{\mathcal P}(A)$ is a principal ideal of $Z$ generated by $d\in Z$,
then $d$ is called the {\it ${\mathcal P}$-discriminant} of $A$, denoted 
by $d_{\mathcal P}(A)$. In this case $d_{\mathcal P}(A)$ is unique up to an element
in $Z^{\times}$.
\item[(5)]
Let ${\mathcal C}$ be a class of algebras over $\Bbbk$. We say that 
${\mathcal P}$ is {\it ${\mathcal C}$-stable} if for every algebra $A$ 
in ${\mathcal C}$ and every $n\geq 1$,
$$I_{\mathcal P}(A\otimes \Bbbk[t_1,\cdots,t_n])=I_{\mathcal P}(A)
\otimes \Bbbk[t_1,\cdots, t_n]$$
as an ideal of $Z\otimes \Bbbk[t_1,\cdots,t_n]$. If ${\mathcal C}$ is a singleton
$\{A\}$, we simply call ${\mathcal P}$ {\it $A$-stable}.
If ${\mathcal C}$ is the whole collection of $\Bbbk$-algebras with the
center affine over $\Bbbk$, we simply call ${\mathcal P}$ {\it stable}.
\end{enumerate}
\end{definition}

In general, neither $L_{\mathcal P}(A)$ nor $D_{\mathcal P}(A)$ 
is a subscheme of $\Spec Z(A)$. 

\begin{example}
\label{xxex2.4} Suppose $\Bbbk={\mathbb C}$. 
Let $A$ be the universal enveloping algebra
of the simple Lie algebra $sl_2$.  It is well known that $Z(A)=\Bbbk[Q]$
where $Q=2(ef + fe) +h^2$. 

Let ${\mathcal S}$ be the property of being simple. Then 
$D_{\mathcal S}(A)$ is the set of integer points of the form
$\{n^2+2n\mid n \in {\mathbb N}\}$ inside the $\MaxSpec \Bbbk[Q]$, 
see \cite{Di} or \cite[p.98]{Sm}. In this case, the ${\mathcal S}$-discriminant 
ideal of $A$ is the zero ideal of $\Bbbk[Q]$ and the 
${\mathcal S}$-discriminant of $A$ is the element $0\in \Bbbk[Q]$.

Note from \cite{Di} or \cite[p.98]{Sm} that for each $c=n^2+2n$, $A/(Q-c)A$ has a 
unique proper two-sided ideal $M_c$ and $M_c$ is of codimension 
$(n+1)^2$. Let ${\mathcal P}_{n}$ be the property of not having a factor 
ring isomorphic to the matrix algebra $M_{n+1}(\Bbbk)$.
Then $D_{{\mathcal P}_n}(A)$ is the singleton $\{n^2+2n\}$, as a subset of
$D_{\mathcal S}(A)$.
As a consequence, the ${\mathcal P}_n$-discriminant ideal of $A$ is 
$(Q-(n^2+2n))\subseteq \Bbbk[Q]$ and the ${\mathcal P}_n$-discriminant of $A$ 
is the element $Q-(n^2+2n)\in \Bbbk[Q]$.

It is clear that ${\mathcal S}$ is a Morita invariant, but 
${\mathcal P}_n$ is not for each fixed $n$.
\end{example}

\begin{lemma}
\label{xxlem2.5} 
Let ${\mathcal P}$ be a property.
\begin{enumerate}
\item[(1)]
Suppose $\phi: A\to B$ is an isomorphism. Then $\phi$ 
preserves the following:
\begin{enumerate}
\item[(1a)]
${\mathcal P}$-locus.
\item[(1b)]
${\mathcal P}$-discriminant set.
\item[(1c)]
${\mathcal P}$-discriminant ideal.
\item[(1d)]
${\mathcal P}$-discriminant {\rm{(}}if it exists{\rm{)}}.
\end{enumerate}
\item[(2)]
Suppose that ${\mathcal P}$ is a  Morita invariant and that $A$ and
$B$ are Morita equivalent. Then the algebra map 
$\omega$ in \eqref{E2.1.1} preserves the following:
\begin{enumerate}
\item[(2a)]
${\mathcal P}$-locus.
\item[(2b)]
${\mathcal P}$-discriminant set.
\item[(2c)]
${\mathcal P}$-discriminant ideal.
\item[(2d)]
${\mathcal P}$-discriminant {\rm{(}}if it exists{\rm{)}}.
\end{enumerate}
\end{enumerate}
\end{lemma}

\begin{proof} (1) Clear.

(2) This follows from the definition, Lemma \ref{xxlem2.1}(5) and the
hypothesis that ${\mathcal P}$ is a Morita invariant.
\end{proof}

In this and the next sections we study two properties
that are closely related to the m-cancellative property. 
The retractable property was introduced in \cite[Definitions 
2.1 and 2.5]{LWZ}. Next we generalize $Z$-retractability
to the Morita setting.

\begin{definition}
\label{xxdef2.6}
Let $A$ be an algebra.
\begin{enumerate}
\item[(1)] \cite[Definition 2.5(1)]{LWZ}
We call $A$ {\it $Z$-retractable}, 
if for any algebra $B$, any algebra isomorphism $\phi: A[t] 
\cong B[s]$ implies that $\phi(Z(A))=Z(B)$.
\item[(2)] \cite[Definition 2.5(2)]{LWZ}
We call $A$ {\it strongly $Z$-retractable}, if for any algebra $B$ 
and integer $n\geq 1$, any algebra isomorphism 
$\phi: A[t_1,\dots, t_n] \cong B[s_1,\dots, s_n]$ 
implies that $\phi(Z(A))=Z(B)$.
\item[(3)]
We call $A$ {\it m-$Z$-retractable} if, for any algebra $B$, any
equivalence of categories $M(A[t])\cong M(B[s])$ implies that 
$\omega(Z(A))=Z(B)$ where $\omega: Z(A)[t]\to Z(B)[s]$ is given as in 
\eqref{E2.1.1}.
\item[(4)] 
We call $A$ {\it strongly m-$Z$-retractable} if, for any algebra $B$ and 
any $n\geq 1$, any equivalence of categories 
$M(A[t_1,\cdots, t_n])\cong M(B[s_1,\cdots,s_n])$ 
implies that $\omega(Z(A))=Z(B)$ where $\omega: Z(A)[t_1,\cdots,t_n]\to 
Z(B)[s_1,\cdots,s_n]$ is given as in \eqref{E2.1.1}.
\end{enumerate}
\end{definition}

The following proposition is similar to \cite[Lemma 2.6]{LWZ}.

\begin{proposition}
\label{xxpro2.7} Let $A$ be an algebra whose center
$Z:=Z(A)$ is an affine domain. Let ${\mathcal P}$ be a stable 
Morita invariant property {\rm{(}}respectively, stable property{\rm{)}}
and assume that the ${\mathcal P}$-discriminant of $A$, denoted by $d$, exists.
\begin{enumerate}
\item[(1)]
Suppose $ML^H_d(Z[t])=Z$. Then
$A$ is m-$Z$-retractable {\rm{(}}respectively, 
$Z$-retractable{\rm{)}}.
\item[(2)]
Suppose that $Z$ is strongly $\LND^H_d$-rigid.
Then $A$ is strongly m-$Z$-retractable {\rm{(}}respectively, strongly 
$Z$-retractable{\rm{)}}. 
\end{enumerate}
\end{proposition}

\begin{proof} The proofs of (1) and (2) are similar, so we  prove only (2). 
We only work on the strongly m-$Z$-retractable version, the strongly 
$Z$-retractable version is similar.

Suppose that $A[t_1,\cdots,t_n]$ is Morita equivalent to 
$B[s_1,\cdots,s_n]$ for some algebra $B$ and for some $n\geq 1$.
Let $\omega: Z\otimes \Bbbk[\underline{t}]\to Z(B)\otimes \Bbbk 
[\underline{s}]$ be the map given in \eqref{E2.1.1}. Since ${\mathcal P}$
is stable, $d_{\mathcal P}(A[\underline{t}])=d\otimes 1$, where
$1$ is the identity element of the polynomial ring
$\Bbbk[\underline{t}]$. In other words, the principal ideal $(d
\otimes 1)$ is the ${\mathcal P}$-discriminant ideal of 
$A[\underline{t}]$. Since $\omega$ preserves the discriminant 
ideal [Lemma \ref{xxlem2.5}(2c)] and ${\mathcal P}$  
is stable, we obtain that 
\begin{equation}
\label{E2.7.1}\tag{E2.7.1}
\omega((d\otimes 1))=\omega((d)\otimes \Bbbk[\underline{t}])
=\omega(I_{\mathcal P}(A[\underline{t}]))=
I_{\mathcal P}(B[\underline{s}])=
I_{\mathcal P}(B)\otimes \Bbbk[\underline{s}].
\end{equation}
As a consequence, $I_{\mathcal P}(B)$ is a principal ideal, denoted by
$(d')$, where $d'$ is the ${\mathcal P}$-discriminant of $B$. 
Equation \eqref{E2.7.1} implies that 
$$\omega(d\otimes 1)=_{Z(B[\underline{s}])^{\times}} d'\otimes 1',$$ 
where $1'$ is the identity element of the polynomial ring $\Bbbk[\underline{s}]$. 
Since $Z(B)$ is a domain, $Z(B[\underline{s}])^{\times}=Z(B)^{\times}$.
Hence $\omega$ maps $d$ to $d'$ up to a scalar in $Z(B)^{\times}$. 

Now consider the map $\omega: Z\otimes \Bbbk[\underline{t}]\to Z(B)
\otimes \Bbbk[\underline{s}]$ again. Since $\omega$ maps $d$
to $d'$, by the strongly $\LND^H_d$-rigidity of $Z$, we have
$$\omega(Z)=\omega(ML^H_d(Z\otimes \Bbbk[\underline{t}]))
=ML^{H}_{d'}(Z(B)\otimes \Bbbk[\underline{s}])
\subseteq Z(B)$$
where the last $\subseteq$ follows from the computation
given in \cite[Example 2.4]{BZ1}. This means that the isomorphism
$\omega$ induces an algebra map from $Z$ to $Z(B)$. Let $Z'$
be the subalgebra $\omega^{-1}(Z(B))\subset Z[\underline{t}]$. 
Then $Z'$ contains $Z$, which is considered as the degree zero part 
of the algebra $Z[\underline{t}]$, and we have
$$\GKdim Z'=\GKdim Z(B)=\GKdim Z(B)[\underline{s}]-n
=\GKdim Z[\underline{t}]-n=\GKdim Z.$$
By \cite[Lemma 3.2]{BZ1}, $Z'=Z$. Therefore $\omega(Z)=Z(B)$ as 
required.
\end{proof}

The rest of this section follows closely \cite[Section 2]{LWZ}.
By \cite[Section 5]{BZ1}, effectiveness (and the dominating property)
of the discriminant controls $\LND^H$-rigidity. We now recall the 
definition of the effectiveness of an element.
An algebra is called {\it PI} if it satisfies a polynomial identity.

Next we will use filtered algebras and associated graded algebras, 
see \cite[Section 1]{YZ2} for more details. By a filtration of a 
$\Bbbk$-algebra $A$ we mean an ascending filtration $F := 
\{F_i A\}_{i\geq 0}$ of vector spaces such that $1\in F_0 A$ and 
$F_i A F_j A\subseteq F_{i+j}A$ for all $i,j\geq 0$. We assume 
that $F$ is (separated and) exhaustive. By \cite[Lemma 1.1]{YZ2}, 
giving a filtration on an algebra $A$ is equivalent to giving 
a degree on the set of generators of $A$. 

\begin{definition} \cite[Definition 5.1]{BZ1}
\label{xxdef2.8}
Let $A$ be a domain and suppose that $Y=\bigoplus_{i=1}^n \Bbbk x_i$
generates $A$ as an algebra.
An element $0 \neq f\in A$ is called \emph{effective} if the following
conditions hold.
\begin{enumerate}
\item[(1)]
There is an ${\mathbb N}$-filtration $\{F_i A\}_{i\geq 0}$ on $A$
such that the associated graded ring $\gr A$ is a domain
(one possible filtration is the trivial filtration $F_0 A=A$).
With this filtration we define the degree of elements in $A$, denoted
by $\deg_A$.
\item[(2)]
For every testing ${\mathbb N}$-filtered PI algebra $T$ with
$\gr T$ being an ${\mathbb N}$-graded domain and for every
testing subset $\{y_1,\ldots,y_{n}\}\subset T$ satisfying
\begin{enumerate}
\item[(a)]
it is linearly independent in the quotient $\Bbbk$-module $T/\Bbbk 1_T$, and
\item[(b)]
$\deg_T y_i\geq \deg_A x_i$ for all $i$ and $\deg_T y_{i_0}>\deg_A x_{i_0}$
for some $i_0$,
\end{enumerate}
then there is a presentation of $f$ of the form
$f(x_1,\ldots,x_{n})$ in the free algebra $\Bbbk\langle x_1,\ldots,x_{n}
\rangle$, such that either $f(y_1,\ldots,y_n)=0$ or
$\deg_{T} f(y_1,\ldots,y_n)>\deg_A f$.
\end{enumerate}
\end{definition}

Here is an easy example.

\begin{example} \cite[Example 2.8]{LWZ}
\label{xxex2.9}
Every non-invertible nonzero element in $\Bbbk[t]$ is effective in $\Bbbk[t]$.
\end{example}

Other examples of effective elements are given in
\cite[Section 5]{BZ1}.
There is another concept, called ``dominating'', see 
\cite[Definition 4.5]{BZ1} or \cite[Definition 2.1(2)]{CPWZ1}, 
that is similar to effectiveness. Both of these properties control 
$\LND^H$-rigidity. The following result is similar to 
\cite[Theorem 5.2]{BZ1} and \cite[Theorem 2.9]{LWZ}.

\begin{theorem}
\label{xxthm2.10} If $d$ is an effective {\rm{(}}respectively,
dominating{\rm{)}} element in an affine 
commutative domain $Z$, then $Z$ is strongly $\LND^H_d$-rigid.
\end{theorem}

\begin{proof}
Since the proofs for the ``effective'' case and the ``dominating''
case are very similar, we  prove only the ``effective'' case. 

Suppose $Z$ is generated by $\{x_j\}_{j=1}^{m}$.
Let $\partial\in \LND^H_d(Z[t_1,\cdots,t_n])$
and $G:=G_{\partial, t}\in \Aut_{\Bbbk [t]}(Z[t_1,\cdots,t_n][t])$ 
as in Definition \ref{xxdef1.6}(2). Then, for each $j$,
$$G(x_j)=x_j+\sum_{i\geq 1} t^i \partial_i(x_j).$$
Since $d \in \ker \partial$, by definition, 
\begin{equation}
\label{E2.10.1}\tag{E2.10.1}
G(d)=d.
\end{equation}
Recall from Definition \ref{xxdef2.8} that, when $d$ is effective, 
$Z$ is a filtered algebra with $\deg_Z$ is defined as in 
\cite[Lemma 1.1]{YZ2}. It is clear that $Z':=Z[t_1,\cdots,t_n]$ 
is a filtered algebra with $\deg_{Z'} z=\deg_Z z$ for all $z\in Z$ 
and $\deg_{Z'} t_s=1$ for $s=1,\cdots,n$. We take the test algebra $T$ 
to be $Z[t_1,\cdots,t_n][t]=Z'[t]$ where the filtration on $T$ is 
determined by $\deg_T(z)=\deg_Z(z)$ for all $z\in Z$, $\deg_T t_s=1$ 
for $s=1,\cdots, n$, and $\deg_T t=\alpha$, where 
$\alpha > \sup \{\deg_{Z'} \partial_i(x_j) \mid j=1,\cdots,m,i\geq 0\}$. 

Now set $y_j=G(x_j)\in T$. By the choice of $\alpha$,
we have that
\begin{enumerate}
\item[(a)]
$\deg_T y_j\geq \deg_Z x_j$, and that
\item[(b)]
$\deg_T y_j=\deg_Z x_j$ if and only if $y_j=x_j$.
\end{enumerate}
Let $f(x_1,\cdots,x_m)$ be some polynomial presentation of $d$
as in Definition \ref{xxdef2.8}. If $G(x_j)\neq x_j$ for some $j$,
by the effectiveness  of $d$ as in Definition \ref{xxdef2.8},  
$f(y_1,\cdots,y_m)=0$ or
$\deg_T f(y_1,\cdots,y_m)>\deg_Z d=\deg_T d$. So $f(y_1,\cdots,y_m)
\neq_{Z^\times} d$. But $f(y_1,\cdots,y_m)=G(d)=_{Z^\times} d$
by \eqref{E2.10.1}, a contradiction. Therefore
$G(x_j)=x_j$ for all $j$. As a consequence,
$\partial_i(x_j)=0$ for all $i\geq 1$, or equivalently, 
$x_j\in \ker \partial$.
Since $Z$ is generated by $x_j$'s, $Z\subset
\ker\partial$. Thus $Z\subseteq ML^H_d(Z[t_1,\cdots,t_n])$.
It is clear that $Z\supseteq ML^H_d(Z[t_1,\cdots,t_n])$,
see \cite[Example 2.4]{BZ1}. Therefore
$Z=ML^H_d(Z[t_1,\cdots,t_n])$ as required.
\end{proof}

The following corollary will be used several times.

\begin{corollary}
\label{xxcor2.11}
Let $A$ be an algebra such that the center of $A$ is $\Bbbk[x]$. 
Let ${\mathcal P}$ be a stable Morita invariant property 
{\rm{(}}respectively, stable property{\rm{)}} such that
the ${\mathcal P}$-discriminant of $A$, denoted by $d$, is a nonzero
non-invertible element in $Z(A)=\Bbbk[x]$. Then $Z(A)$ is strongly 
$\LND^H_d$-rigid and $A$ is strongly m-$Z$-retractable 
{\rm{(}}respectively, strongly $Z$-retractable{\rm{)}}.
\end{corollary}

\begin{proof} By Example \ref{xxex2.9}, $d$ is an effective element
in $Z(A)$. By Theorem \ref{xxthm2.10}, $Z(A)$ is strongly $\LND^H_d$-rigid.
By Proposition \ref{xxpro2.7}(2), $A$ is strongly m-$Z$-retractable
{\rm{(}}respectively, strongly $Z$-retractable{\rm{)}}. 
\end{proof}

\section{Morita Detectability}
\label{xxsec3}

First we recall the detectability introduced in \cite{LWZ}.
If $B$ is a subring of $C$ and $f_1,\dots,f_m$ are elements of $C$, then the
subring generated by $B$ and the subset $\{f_1,\dots,f_m\}$ is denoted by 
$B\{f_1,\dots,f_m\}$.

\begin{definition} \cite[Definition 3.1]{LWZ}
\label{xxdef3.1} Let $A$ be an algebra.
\begin{enumerate}
\item[(1)]
We call $A$ {\it detectable}
if any algebra isomorphism $\phi: A[t] \cong B[s]$ for some algebra $B$ implies 
that $B[s]=B\{\phi(t)\}$, or equivalently, $s\in B\{\phi(t)\}$.
\item[(2)]
We call $A$ {\it strongly detectable} if for each
integer $n\geq 1$ and any algebra isomorphism
$$\phi: A[t_1,\dots, t_n] \cong B[s_1,\dots, s_n]$$
 for some algebra $B$ 
implies that $B[s_1,\dots,
s_n]=B\{\phi(t_1),\dots,\phi(t_n)\}$, or equivalently, for each 
$i=1,\cdots, n$, $s_i\in B\{\phi(t_1),\dots,\phi(t_n)\}$.
\end{enumerate}
\end{definition}

In the above definition, we do not assume that $\phi(t)=s$.
Every strongly detectable algebra is detectable.
The polynomial ring $\Bbbk[x]$ is cancellative, but not
detectable. By \cite[Lemma 3.2]{LWZ}, if $A$ is $Z$-retractable 
in the sense of \cite[Definition 2.5]{LWZ}, then it is detectable. 
We first recall a definition from \cite[Definition 3.4]{LWZ}.

\begin{definition} \cite[Definition 3.4]{LWZ}
\label{xxdef3.2}
Let $A$ be an algebra over $\Bbbk$.
\begin{enumerate}
\item[(1)]
We say $A$ is {\it Hopfian} if every $\Bbbk$-algebra 
epimorphism from $A$ to itself
is an automorphism.
\item[(2)]
We say $A$ is {\it strongly Hopfian} if $A[t_1,\cdots, t_n]$ 
is Hopfian for every $n\geq 0$.
\end{enumerate}
\end{definition}

By \cite[Lemma 3.6]{LWZ}, if $A$ is detectable and strongly Hopfian,
then $A$ is cancellative. We will generalize these facts in the 
Morita setting. In the following definition, we use $\omega^{-1}$ 
instead of $\omega$ for some technical reasons.

\begin{definition} 
\label{xxdef3.3} Let $A$ be an algebra. Let $\omega$ be the 
map given in \eqref{E2.1.1} when in a Morita context. 
\begin{enumerate}
\item[(1)]
We call $A$ {\it m-detectable} if any equivalence of categories
$M(A[t]) \cong M(B[s])$ for some algebra $B$ implies that
$$A[t]=A\{\omega^{-1}(s)\},$$ 
or equivalently, $t\in A\{\omega^{-1}(s)\}$.
\item[(2)]
We call $A$ {\it strongly m-detectable} if for each  $n\geq 1$ and
any equivalence of categories $M(A[t_1,\cdots,t_n]) \cong M(B[s_1,\cdots,s_n])$
for some algebra $B$ implies that
$$A[t_1,\cdots,t_n]=A\{\omega^{-1}(s_1),\cdots,\omega^{-1}(s_n)\},$$ 
or equivalently, $t_i\in A\{\omega^{-1}(s_1),\cdots,\omega^{-1}(s_n)\}$ for $i
=1,\cdots,n$.
\end{enumerate}
\end{definition}

The following result is analogous to \cite[Lemma 3.2]{LWZ}.

\begin{lemma}
\label{xxlem3.4} If $A$ is m-$Z$-retractable {\rm{(}}respectively, strongly
m-$Z$-retractable{\rm{)}}, then it is m-detectable {\rm{(}}respectively, strongly
m-detectable{\rm{)}}.
\end{lemma}

\begin{proof} We  show only the ``strongly'' version.
% The proof of the non-``strongly'' version is similar.

Suppose that $A$ is strongly m-$Z$-retractable. Let $B$ be any algebra 
such that the abelian categories $M(A[\underline{t}])$ and $M(B[\underline{s}])$ 
are equivalent. Since $A$ is strongly m-$Z$-retractable, the map 
$\omega: Z(A)[\underline{t}]\to Z(B)[\underline{s}]$ in \eqref{E2.1.1} restricts 
to an algebra isomorphism $Z(A)\to Z(B)$. Write
$\phi=\omega^{-1}$ and $f_i:=\phi(s_i)$ for $i=1,\dots,n$. Then
$$\begin{aligned}
Z(A)\{f_1,\cdots,f_n\}&=\phi(Z(B))\{\phi(s_1),\dots, \phi(s_n)\}\\
&=\phi(Z(B)\{s_1,\dots,s_n\})\\ 
&=\phi(Z(B)[\underline{s}])\\
%&=\phi (Z(B[\underline{s}]))=Z(A[\underline{t}])\\
&= Z(A)[\underline{t}].
\end{aligned}
$$
Then, for every $i$, $t_i\in Z(A)[\underline{t}]= Z(A)\{f_1,\cdots,f_n\} \subseteq
A\{f_1,\dots,f_n\}$ as desired.
\end{proof}

Next we show that m-detectability implies m-cancellative property 
under some mild conditions. 

\begin{example} \cite[Lemma 3.5]{LWZ}
\label{xxex3.5}
The following algebras are strongly Hopfian.
\begin{enumerate}
\item[(1)]
Left or right noetherian algebras.
\item[(2)]
Finitely generated locally finite ${\mathbb N}$-graded algebras.
\item[(3)]
Prime affine $\Bbbk$-algebras satisfying a polynomial identity.
\end{enumerate}
\end{example}

\begin{lemma}
\label{xxlem3.6} Suppose $A$ is strongly Hopfian.
\begin{enumerate}
\item[(1)]
If $A$ is m-detectable, then $A$ is m-cancellative 
 and  cancellative.
\item[(2)]
If $A$ is strongly m-detectable, then $A$ is strongly m-cancellative
and strongly cancellative.
\end{enumerate}
\end{lemma}

\begin{proof} We  prove only (2).

First we consider the Morita version. Suppose that $A[\underline{t}]$ 
and $B[\underline{s}]$ are Morita equivalent
and $\omega: Z(A)[\underline{t}]\to Z(B)[\underline{s}]$ is the algebra 
isomorphism given as in \eqref{E2.1.1}. Write $\phi=\omega^{-1}$ and
$f_i=\phi(s_i)$ for $i=1,\cdots, n$. Then $f_i$ are central elements
in $A[\underline{t}]$. Thus $A\{f_1,\cdots,f_n\}$ is a homomorphic image
of $A[t_1,\cdots,t_n]$ by sending $t_i\mapsto f_i$. Since $A$ is strongly
m-detectable. Then $A\{f_1,\cdots, f_n\}=A[\underline{t}]$.
Then we have an algebra homomorphism
\begin{equation}
\label{E3.6.1}\tag{E3.6.1}
A[\underline{t}]\xrightarrow{\pi} A\{f_1,\cdots,f_n\}\xrightarrow{=} 
A[\underline{t}].
\end{equation}
Since $A$ is strongly Hopfian, $A[\underline{t}]$ is Hopfian. Now 
\eqref{E3.6.1} implies that $\pi$ is an isomorphism.
As a consequence, $A\{f_1,\cdots, f_n\}=A[f_1,\cdots,f_n]$ viewing $f_i$ as central
indeterminates in $A[f_1,\cdots,f_n]$. As a consequence $A[\underline{t}]=
A[\underline{f}]$. Going back to the map
$$\omega: Z(A[\underline{t}])=Z(A[\underline{f}])\to Z(B[\underline{s}]),$$
one sees that $\omega$ maps $f_i$ to $s_i$ for $i=1,\cdots,n$.
Let $J$ be the ideal of $Z(A[\underline{t}])$ generated by $\{f_i\}_{i=1}^n$
and $J'$ be the ideal of $Z(B[\underline{s}])$ generated by $\{s_i\}_{i=1}^n$.
Then $J'=\omega(J)$. By Lemma \ref{xxlem2.1}(5), the algebra
$A$ (which is isomorphic to $A[\underline{t}]/JA$) is Morita
equivalent to $B$ (which is isomorphic to $B[\underline{s}]/J'B$). 
The assertion follows.

Next we consider the ``cancellative" version.
Suppose that $\omega':A[\underline{t}]\to B[\underline{s}]$ is an isomorphism
which restricts to an isomorphism between the centers 
$\omega: Z(A)[\underline{t}]\to Z(B)[\underline{s}]$. Then $\omega'$
induces a (trivial) Morita equivalence, and $\omega$ is the map 
given in \eqref{E2.1.1}.
Re-use the notation introduced in the above proof. The above proof shows that
$A[\underline{t}]=A[\underline{f}]$ where $f_i=\omega^{-1}(s_i)$ for all
$i$. Therefore $\omega'$ induces isomorphism
$$A\cong A[\underline{f}]/(\{f_i\}_{i=1}^n)\xrightarrow{\overline{\omega'}}
B[\underline{s}]/(\{s_i\}_{i=1}^n)\cong B$$
as desired. 
\end{proof}

For the rest of this section we study more properties concerning 
m-detectability.

\begin{lemma}
\label{xxlem3.7}
Let $A$ be an algebra with center $Z$. Suppose $Z$ is {\rm{(}}strongly{\rm{)}} 
cancellative.
\begin{enumerate}
\item[(1)]
If $Z$ is {\rm{(}}strongly{\rm{)}} detectable, then $A$ is {\rm{(}}strongly{\rm{)}}
m-detectable.
\item[(2)]
$Z$ is {\rm{(}}strongly{\rm{)}} detectable if and only if it is {\rm{(}}strongly{\rm{)}}
m-detectable.
\end{enumerate}
\end{lemma}

\begin{proof} Following the pattern before, we  prove only the ``strongly''
version.

(1) Suppose $B$ is an algebra such that $A[\underline{t}]$ and $B[\underline{s}]$ 
are Morita equivalent. Let $\omega: Z[\underline{t}]\to Z(B)[\underline{s}]$ 
be the algebra isomorphism given in \eqref{E2.1.1}. Since $Z$ is strongly cancellative,
one has that $Z(B)\cong Z$. Now we have an isomorphism 
$\omega^{-1}: Z(B)[\underline{s}]\cong 
Z[\underline{t}]$. Since $Z(B)$ (or $Z$) is strongly detectable, 
$t_i\in Z\{\omega^{-1}(s_1), \cdots, \omega^{-1}(s_n)\}$ for 
all $i$. Thus $t_i\in A\{\omega^{-1}(s_1), \cdots, \omega^{-1}(s_n)\}$ 
for all $i$. This means that $A$ is strongly
m-detectable.

(2) One direction is part (1). For the other direction, assume that 
$Z$ is strongly m-detectable. Consider any algebra isomorphism
$\phi: Z[\underline{t}]\to B[\underline{s}]$. It is clear that $B$ is 
commutative and $B\cong Z$ since $Z$ is strongly cancellative.
Then $\phi$ induces a (trivial) Morita equivalent and the map
$\omega$ in \eqref{E2.1.1} is just $\phi$. Now the strong
$m$-detectability of $Z$ implies that $Z$ is strongly detectable. 
\end{proof}

The next result is similar to \cite[Proposition 3.10]{LWZ}.

\begin{proposition}
\label{xxpro3.8}
If the center $Z$ of $A$ is an affine domain of GK-dimension one
that is not isomorphic to $\Bbbk'[x]$ for some field extension
$\Bbbk'\supseteq \Bbbk$, then $A$ is  strongly
m-detectable.
\end{proposition}

\begin{proof} By \cite[Theorem 3.3]{AEH}, $Z$ is strongly retractable 
and cancellative. As a consequence, $Z$ is a strongly m-$Z$-retractable.
By Lemma \ref{xxlem3.4}, $A$ is strongly m-detectable. 
\end{proof}

\section{Proofs of Theorems \ref{xxthm0.3} and \ref{xxthm0.4}}
\label{xxsec4}

In this section we will use the results in the previous
sections to show some classes of algebras are m-cancellative.
We first prove Theorem \ref{xxthm0.4}.

\begin{theorem}
\label{xxthm4.1} If A is left (or right) noetherian, and the 
center of A is artinian, then $A$ is strongly m-detectable. As 
a consequence, $A$ is strongly m-cancellative.
\end{theorem}

\begin{proof} Let $Z$ be the center of $A$. Then $Z$ is artinian
by hypothesis. By \cite[Theorem 4.1]{LWZ}, $Z$ is strongly 
detectable and strongly cancellative. By Lemma \ref{xxlem3.7}(1), 
$A$ is strongly m-detectable. By Example \ref{xxex3.5}(1), $A$ is 
strongly Hopfian. The consequence follows from Lemma \ref{xxlem3.6}(2). 
\end{proof}

Theorem \ref{xxthm0.4} is a special case of Theorem \ref{xxthm4.1}.

\begin{theorem}
\label{xxthm4.2}
Let $A$ be an algebra with  strongly cancellative center $Z$. 
Suppose $J$ is the prime radical of $Z$ such that {\rm{(a)}} 
$J$ is nilpotent and {\rm{(b)}} $Z/J$ is a finite direct sum of
fields. Then the following hold. 
\begin{enumerate}
\item[(1)]
$A$ is strongly m-detectable.
\item[(2)]
If further $A$ is strongly Hopfian, then $A$ is strongly m-cancellative.
\end{enumerate}
\end{theorem}

\begin{proof}
(1) By the proof of \cite[Theorem 4.2]{LWZ}, $Z$ is strongly detectable. 
By Lemma \ref{xxlem3.7}, $A$ is strongly m-detectable.

(2) Follows from Lemma \ref{xxlem3.6} and part (1).
\end{proof}

Next is Theorem \ref{xxthm0.3}.

\begin{corollary}
\label{xxcor4.3}
Suppose $A$ is strongly Hopfian 
and the center of $A$ is artinian. Then $A$ is strongly m-detectable 
and strongly m-cancellative.
\end{corollary}

\begin{proof} Let $Z$ be the center of $A$. By \cite[Theorem 4.1]{LWZ}, 
$Z$ is strongly detectable and strongly cancellative. Since 
$Z$ is artinian, it satisfies conditions (a) and (b) in Theorem 
\ref{xxthm4.2}. The assertion follows by Theorem \ref{xxthm4.2}.
\end{proof}

\section{Proof of Theorem \ref{xxthm0.6}}
\label{xxsec5}

{\bf We assume that $\Bbbk$ is algebraically closed in this section.} 
Under this hypothesis, a ${\mathcal P}$-discriminant ideal
has the following nice property. This is one of the reasons we need the 
above hypothesis.

\begin{lemma}
\label{xxlem5.1} Let ${\mathcal P}$ be a property. 
Then ${\mathcal P}$ is stable.
\end{lemma}

\begin{proof} Let $Z$ be the center of $A$. By Definition \ref{xxdef2.3}(5),
we may assume that $Z$ is affine and write it as $\Bbbk[z_1,\cdots,z_m]/(R)$
where $\{z_1,\cdots,z_m\}$ is a generating set of $Z$ and $R$ is a set of 
relations. Every maximal ideal of $Z$ is of the form $(z_i-\alpha_i):=
(z_1-\alpha_1,\cdots,z_m-\alpha_m)$, where $\alpha_i\in \Bbbk$ for all $i$. 
Every maximal ideal of $Z[\underline{t}]$ is of the form
$$(z_i-\alpha_i,t_j-\beta_j)
:=(z_1-\alpha_1,\cdots,z_m-\alpha_m, t_1-\beta_1,\cdots,t_n-\beta_n)$$
where $\alpha_i, \beta_j\in \Bbbk$.  The natural embedding 
$Z\to Z[\underline{t}]$ induces a projection
$$\pi: \MaxSpec(Z[\underline{t}])\to \MaxSpec Z$$
by sending $\fm:=(z_i-\alpha_i,t_j-\beta_j)$ to $\pi(\fm):=(z_i-\alpha_i)$. 

Let $D_{\mathcal P}(A)$ be the ${\mathcal P}$-discriminant set of $A$. 
A maximal ideal $\fm$ is in $D_{\mathcal P}(A[\underline{t}])$ if and only
if $A[\underline{t}]/\fm A[\underline{t}]$ does not have property ${\mathcal P}$. 
Since $A[\underline{t}]/\fm A[\underline{t}] \cong A/\pi(\fm)A$, 
$\fm\in D_{\mathcal P}(A[\underline{t}])$ if and only if $\pi(\fm)
\in D_{\mathcal P}(A)$. This implies that $D_{\mathcal P}(A[\underline{t}])
=D_{\mathcal P}(A) \times {\mathbb A}^n$. As a consequence,
$$I_{\mathcal P}(A[\underline{t}])=
\bigcap_{\fm\in D_{\mathcal P}(A[\underline{t}])} \fm=
\big( \bigcap_{p\in D_{\mathcal P}(A)} p\big) \otimes \Bbbk[\underline{t}]
=I_{\mathcal P}(A)\otimes \Bbbk[\underline{t}].$$
Therefore ${\mathcal P}$ is stable by Definition \ref{xxdef2.3}(5).
\end{proof}

Let $A$ be an algebra with the center $Z$ being a domain. Let $\tau(A/Z)$ be the 
ideal of $A$ consisting of elements in $A$ that are annihilated by some nonzero
element in $Z$. Define the {\it annihilator ideal} of $Z$ to be
$$\kappa(A/Z)=\{z\in Z\mid z(\tau(A/Z))=0\}.$$

\begin{lemma}
\label{xxlem5.2} Retain the notation as above.
\begin{enumerate}
\item[(1)]
$\kappa$ is stable in the sense that $\kappa(A[\underline{t}]/Z[\underline{t}])
=\kappa(A/Z)\otimes \Bbbk[\underline{t}]$.
\item[(2)]
If $A$ and $B$ are Morita equivalent, then $\omega$ maps 
$\kappa(A/Z)$ to $\kappa(B/Z(B))$ bijectively.
\item[(3)]
If $A$ is left noetherian and suppose the center 
$Z$ is a domain, then $\tau(A/Z)\neq 0$ if and only if 
$\kappa(A/Z)$ is a proper ideal, neither $Z$ nor 0.
\end{enumerate}
\end{lemma}

\begin{proof} Easy to check. Details are omitted.
\end{proof}

\begin{lemma}
\label{xxlem5.3} 
Suppose $A$ is a finitely generated module over its center $Z$ and 
$Z$ is a domain. If $A$ is prime, then $\tau(A/Z)=0$.
\end{lemma}

\begin{proof} Easy to check. Details are omitted.
\end{proof}

\begin{proposition}
\label{xxpro5.4} 
Let $A$ be left noetherian such that the center $Z$ is 
an affine domain of GK-dimension one.
\begin{enumerate}
\item[(1)]
If $Z$ is not $\Bbbk[x]$, then $A$ is strongly m-$Z$-retractable, 
m-detectable, and m-cancellative.
\item[(2)]
If $Z=\Bbbk[x]$ and $\tau(A/Z)\neq 0$, then $A$ is 
strongly m-$Z$-retractable, m-detectable, and m-cancellative.
\end{enumerate}
\end{proposition}

\begin{proof} (1) By \cite[Theorem 3.3 and Corollary 3.4]{AEH},
$Z$ is strongly retractable. By Definition \ref{xxdef2.6}(3),
$A$ is strongly $m$-$Z$-retractable. By Lemma \ref{xxlem3.4},
$A$ is strongly m-detectable. Since $A$ is left noetherian, by 
Lemma \ref{xxlem3.6}(2), $A$ is strongly m-cancellative.

(2) Since $A$ is left noetherian and $\tau(A/Z)\neq 0$,
$\kappa(A/Z)$ is a nonzero proper ideal of $\Bbbk[x]$ by Lemma 
\ref{xxlem5.2}(3).
So there is a nonzero non-invertible element $f\in 
\Bbbk[x]$ such that $\kappa(A/Z)=(f)$. By Lemma \ref{xxlem5.2}(1,2),
$\kappa$ is a stable Morita invariant property. By replacing
${\mathcal P}$ by $\kappa$, Corollary \ref{xxcor2.11} 
implies that $A$ is strongly m-$Z$-retractable. The rest of the proof
is similar to the proof of part (1).
\end{proof}

For the rest of the section we consider the case when $Z=\Bbbk[x]$
and $\tau(A/Z)=0$, or more precisely, when $A$ is affine prime PI 
of GK-dimension one with $Z=\Bbbk[x]$. We need to recall some concepts.

Let $A$ be an affine prime algebra of GK-dimension one.
By a result of Small-Warfield \cite{SW}, $A$ is a finitely generated
module over its affine center. As a consequence, $A$ is noetherian.

Let $R$ be a commutative algebra, an $R$-algebra $A$ is called 
{\it Azumaya} if $A$ is a finitely generated faithful projective 
$R$-module and the canonical morphism
\begin{equation}
\label{E5.4.1}\tag{E5.4.1}
 A\otimes_R A^{op}\to \End_R(A)
\end{equation}
is an isomorphism. By \cite[Theorem 3.4]{DeI},
$A$ is Azumaya if and only if $A$ is a central separable algebra over $R$.
Since we assume that $\Bbbk$ is algebraically 
closed, we have the following equivalent definition.

\begin{definition} \cite[Introduction]{BY}
\label{xxdef5.5}
Let $A$ be an affine prime $\Bbbk$-algebra which is a finitely 
generated module over its affine center $Z(A)$. Let $n$ be the 
PI-degree of $A$, which is also the maximal possible 
$\Bbbk$-dimension of irreducible $A$-modules.
\begin{enumerate}
\item[(1)]
The {\it Azumaya locus} of $A$, denoted by ${\mathcal A}(A)$,
%$${\mathcal A}(A)\subseteq {\rm{Maxspec}}\; Z(A),$$
is the dense open subset of ${\rm{Maxspec}}\; Z(A)$ which
parametrizes the irreducible $A$-modules of maximal $\Bbbk$-dimension.
In other words, $\fm \in {\mathcal A}(A)$ if and only if
$\fm A$ is the annihilator in $A$ of an irreducible $A$-module
$V$ with $\dim V = n$, if and only if $A/\fm A\cong M_{n}(\Bbbk)$.
\item[(2)]
If ${\mathcal A}(A)={\rm{Maxspec}}\; Z(A)$, $A$ is called
{\it Azumaya}.
\end{enumerate} 
\end{definition}

We can relate the Azumaya locus with the ``simple''-locus. 
Let ${\mathcal S}$ be the property of being simple.

\begin{lemma}
\label{xxlem5.6} Assume that $A$ is free over its affine 
center $Z$.
\begin{enumerate}
\item[(1)]
$A[\underline{t}]$ is free over $Z[\underline{t}]$.
\item[(2)]
${\mathcal A}(A)=L_{\mathcal S}(A)$ where the latter is defined 
in Definition \ref{xxdef2.3}(1).
\end{enumerate}
\end{lemma}

\begin{proof} (1) is obvious.

(2) Since $A$ is free over $Z$ of rank $n^2$, $A/\fm A$ is
isomorphic to $M_n(\Bbbk)$ if and only if $A/\fm A$ is 
simple. The assertion follows.
\end{proof}

\begin{proposition}
\label{xxpro5.7}
Suppose that $A$ is an affine prime algebra of GK-dimension one
with center $\Bbbk[x]$.
\begin{enumerate}
\item[(1)]
If $A$ is not Azumaya, then $A$ is strongly m-$Z$-retractable, 
m-detectable, and m-cancellative.
\item[(2)]
If $A$ is Azumaya, then $A$ is strongly m-cancellative.
\end{enumerate}
\end{proposition}

\begin{proof} (1) Since the Azumaya locus is open and dense, the 
non-Azumaya locus of $A$ is a proper nonzero ideal of $Z=\Bbbk[x]$,
which is principal.
Since $A$ is prime, $\tau(A/Z)=0$ and whence $A$ is projective and then
free over $Z$. By Lemma \ref{xxlem5.6}(2), the Azumaya locus of $A[\underline{t}]$
agrees with the ${\mathcal S}$-locus of $A[\underline{t}]$.
Hence ${\mathcal S}$ is a stable Morita invariant 
property such that the ${\mathcal S}$-discriminant is a nonzero
non-invertible element in $Z$. By Corollary \ref{xxcor2.11},
$A$ is strongly m-$Z$-retractable. The rest of the proof follows
from the proof of Proposition \ref{xxpro5.4}(1).

(2) Since $A$ is Azumaya, by \cite[Lemma 4.9(3)]{LWZ}, $A=
M_n(\Bbbk[x])$ for some integer $n\geq 1$. If 
$A[\underline{t}]$ is Morita equivalent to $B[\underline{s}]$, 
then $Z(A)[\underline{t}]\cong Z(B)[\underline{s}]$. Since $Z(A)=\Bbbk[x]$ 
is strongly cancellative, $Z(B)$ is also isomorphic to $\Bbbk[x]$.
If $B$ is not Azumaya, it follows from part (1) that 
$A$ and $B$ are Morita equivalent. If $B$ is Azumaya, then 
by \cite[Lemma 4.9(3)]{LWZ}, $B$ is a 
matrix algebra $M_{n'}(\Bbbk[x])$ for some $n'\geq 1$, which is also 
Morita equivalent to $A$. Therefore $A$ is strongly 
m-cancellative.
\end{proof}

Now we are ready to prove Theorem \ref{xxthm0.6}.

\begin{theorem}
\label{xxthm5.8}
Let $A$ be an affine prime algebra of GK-dimension one.
\begin{enumerate}
\item[(1)]
$A$ is strongly m-cancellative.
\item[(2)]
If either $Z(A)\neq \Bbbk[x]$ or
$A$ is not Azumaya, then $A$ is
strongly m-$Z$-retractable and m-detectable.
\end{enumerate}
\end{theorem}

\begin{proof} Since we assume that $\Bbbk$ is algebraically closed in
this section, by \cite[Lemma 4.9]{LWZ}, there are three cases
to consider.

Case 1: $Z(A)\not\cong \Bbbk[x]$. 

Case 2: $Z(A)\cong \Bbbk[x]$ and $A$ is not Azumaya.

Case 3: $Z(A)\cong \Bbbk[x]$ and $A$ is Azumaya. 

Applying Proposition \ref{xxpro5.4}(1) in Case 1, 
Proposition \ref{xxpro5.7}(1) in Case 2 and Proposition 
\ref{xxpro5.7}(2) in Case 3, the assertion follows.
\end{proof}

It is clear that Theorem \ref{xxthm0.6} is an immediate 
consequence of Theorem \ref{xxthm5.8}.
As far as we know there are no examples of algebras 
with the center being an affine domain of GK-dimension one 
that are not m-cancellative. Therefore we ask

\begin{question}
\label{xxque5.9} 
Let $A$ be a left noetherian algebra such that 
$Z(A)$ is an affine domain of GK-dimension one. 
Then is $A$ m-cancellative?
\end{question}

We finish this section with some examples of non-PI algebras
that are strongly (m-)cancellative.

\begin{example}
\label{xxex5.10} 
Let $Z$ denote the center of the given algebra $A$.
Assume that $\Bbbk$ has characteristic zero.
\begin{enumerate}
\item[(1)]
Let $A$ be the homogenization of the first Weyl algebra 
which is generated by $x,y,t$ subject to the relations
$$[x,t]=[y,t]=0, [x,y]=t^2.$$ It is well-known that 
the center of $A$ is $\Bbbk[t]$. Let ${\mathcal S}$ be the
property of being simple. Since $\fm:=(t-0)$ is the 
only maximal ideal of $\Bbbk[t]$ such that  $A/\fm A$ is not
simple, the ${\mathcal S}$-discriminant $d_{\mathcal S}(A)$ 
exists and equals $t$.
By Corollary \ref{xxcor2.11}, $A$ is strongly m-$Z$-retractable.
By Lemma \ref{xxlem3.4}, $A$ is strongly m-detectable. By Lemma 
\ref{xxlem3.6}(2), $A$ is both strongly cancellative and 
strongly m-cancellative.
\item[(2)]
Let $A$ be a non-PI quadratic Sklyanin algebra of global dimension 3.
It is well-known that the center of $A$ is $\Bbbk[g]$ where $g\in A$ 
has degree 3. We claim that $A/(g-\alpha)$ is simple if and only if 
$\alpha\neq 0$. If $\alpha=0$, then $A/(g)$ is connected graded which 
is not simple. Now assume that $\alpha\neq 0$. It is well-known that 
$(A[g^{-1}])_0$ is simple. Let $T$ be the 3rd Veronese subring of $A$. 
Then $(T[g^{-1}])_0\cong (A[g^{-1}])_0$ is simple. Now 
$$T/(g-\alpha)\cong T/(\alpha^{-1}g-1)\cong (T[(\alpha^{-1}g)^{-1}])_0
= (T[g^{-1}])_0\cong (A[g^{-1}])_0$$
where the second $\cong$ is due to \cite[Lemma 2.1]{RSS}. It is clear
that $A/(g-\alpha)$ contains $T/(g-\alpha)$. Since  $T/(g-\alpha)$
is simple and hence has no finite dimensional modules, $A/(g-\alpha)$ 
does not have finite dimensional modules. Since the algebra
$A/(g-\alpha)$ is 
affine of GK-dimension two, it must be simple. So we proved the claim.
The claim implies that the ${\mathcal S}$-discriminant 
$d_{\mathcal S}(A)$ exists and equals $g\in \Bbbk[g]$. Following the 
last part of the above example, $A$ is both strongly cancellative and 
strongly m-cancellative.
\end{enumerate}
\end{example}

\begin{example}
\label{xxex5.11} 
Suppose  ${\text{char}}\; \Bbbk=0$. 
Let $A$ be the universal enveloping algebra of the simple Lie
algebra $sl_2$. By Example \ref{xxex2.4}, the center of 
$A$ is $\Bbbk[Q]$ where $Q$ is the Casimir element. In this 
example, we will consider two different properties.

Let ${\mathcal W}$ be the property of not having a factor 
ring isomorphic to $M_{n+1}(\Bbbk)$ (for a fixed integer $n$). 
Then $d_{\mathcal W}(A)=Q-(n^2+2n)$ which is a nonzero non-invertible 
element in $\Bbbk[Q]$. By Corollary \ref{xxcor2.11}, $A$ is strongly 
$Z$-retractable. By \cite[Lemma 3.2]{LWZ}, $A$ is strongly 
detectable, and by \cite[Lemma 3.6(2)]{LWZ}, $A$ is strongly 
cancellative.

Next we show that $A$ is strongly m-cancellative by using a
Morita invariant property. Let ${\mathcal H}$ be the property 
that $HH_3(R)= 0$ where $HH_i(R)$ denotes the $i$th Hochschild 
homology of an algebra $R$. By \cite[Theorem 9.5.6]{We}, the 
Hochschild homology is Morita invariant. Hence ${\mathcal H}$ 
is Morita invariant. We claim that the discriminant 
$d_{\mathcal H}(A)$ is $Q+\frac{1}{4}$. This claim is equivalent 
to the following assertions
\begin{enumerate}
\item[(a)]
$HH_3(A/(Q-\lambda))=0$ for all $\lambda\neq -\frac{1}{4}$; and
\item[(b)]
$HH_3(A/(Q+\frac{1}{4}))\neq 0$ (this is the case when 
$\lambda=-\frac{1}{4}$).
\end{enumerate}
Let $B_{\lambda}=A/(Q-\lambda)$. Then $B_{\lambda}$ agrees with 
the algebra $B_{\lambda}$ in \cite[Example 2.3]{FSS}. By 
\cite[Example 2.3]{FSS}, $B_{\lambda}$ is a generalized Weyl 
algebra with $\sigma(h)=h-1$, $a=\lambda-h(h+1)$. Hence 
$B_{\lambda}$ satisfies the hypotheses of 
\cite[Theorem 2.1]{FSS}. If $\lambda\neq -\frac{1}{4}$, 
then $a'(h)$ and $a(h)$ are coprime. By \cite[Theorem 2.1(1)]{FSS},
$HH_3(B_{\lambda})=0$, which is part (a). If $\lambda= -\frac{1}{4}$, 
then the common divisor of $a'(h)$ and $a(h)$ is $a'(h)$ which 
has degree 1. By \cite[Theorem 2.1(2)]{FSS},
$HH_3(B_{\lambda})=\Bbbk$, which is part (b). Therefore 
we proved the claim. By Corollary \ref{xxcor2.11}, $A$ is 
strongly m-$Z$-retractable.
By Lemma \ref{xxlem3.4}, $A$ is strongly m-detectable. By Lemma 
\ref{xxlem3.6}(2), $A$ is both strongly cancellative and 
strongly m-cancellative. 
\end{example}

\begin{remark}
\label{xxrem5.12}
\begin{enumerate}
\item[(1)]
The second half of Example \ref{xxex5.11} shows that using a Morita 
invariant property results a better conclusion.
\item[(2)]
Another consequence of the discussion in Example \ref{xxex5.11}
is the following. If $\sigma$ is an algebra automorphism of $A:=U(sl_2)$, 
then $\sigma(Q)=Q$. Further, for every locally nilpotent derivation
$\partial\in \LND(A)$, we have $\partial(Q)=0$. This could be a useful 
fact to use in calculating the automorphism group $\Aut(A)$. According 
to \cite[Section 3.2]{CL}, the full automorphism group of $A$ is 
still unknown. A result of Joseph \cite{Jo} says that $\Aut(A)$ 
contains a wild automorphism. The automorphism of $A/(Q-\alpha)A$
was computed in \cite{Di} when $\alpha\neq n^2+2n$ for all $n\in 
{\mathbb N}$.
\end{enumerate}
\end{remark}

\section{Proof of Theorem \ref{xxthm0.5}}
\label{xxsec6}

In this section we prove Theorem \ref{xxthm0.5}. We refer to \cite{ASS} 
for basic definitions of quivers and their path algebra. Let $C_n$ be 
the cyclic quiver with $n$ vertices and $n$ arrow connecting these 
vertices in one oriented direction. In representation theory of finite 
dimensional algebras, quiver $C_n$ is also called {\it type 
$\widetilde{A}_{n-1}$}. Let $0,1,\dots, n-1$ be the vertices of $C_n$, 
and $a_i: i\to i+1$ (in ${\mathbb Z}/(n)$) be the arrows in $C_n$. Then 
$w:= \sum_{i=0}^{n-1} a_{i}a_{i+1}\cdots a_{i+n-1}$ is a central element 
in $\Bbbk C_n$. By \cite[Lemma 4.4]{LWZ}, we have the following result
concerning the center of the path algebra $\Bbbk Q$ when $Q$ is connected:
\begin{equation}
\label{E6.0.1}\tag{E6.0.1}
Z(\Bbbk Q)=\begin{cases}
\Bbbk & {\text{if $Q$ has no arrow,}}\\
\Bbbk[x]& {\text{if $Q=C_1$ or equivalently $\Bbbk Q=\Bbbk[x]$, }}\\
\Bbbk[w]& {\text{if $Q=C_n$ for $n\geq 2$,}}\\
\Bbbk &{\text{otherwise.}}
\end{cases}
\end{equation}

Similar to \cite[Lemma 3.11]{LWZ}, we have the following, and 
its proof is omitted.

\begin{lemma}
\label{xxlem6.1}
Let $\Bbbk'$ be a field extension of $\Bbbk$.
If $A\otimes_{\Bbbk} \Bbbk'$ is {\rm{(}}strongly{\rm{)}} $m$-detectable
as an algebra over $\Bbbk'$, then $A$ is {\rm{(}}strongly{\rm{)}}
$m$-detectable as an algebra over $\Bbbk$.
\end{lemma}

\begin{lemma}
\label{xxlem6.2}
Let $Q=C_n$ for $n\geq 2$. Then 
$\Bbbk Q$ is strongly m-detectable and strongly m-cancellative.
\end{lemma}

\begin{proof} By \cite[Lemma 4.5]{LWZ}, $\Bbbk C_n$ is prime of 
GK-dimension one while not being Azumaya. If $\Bbbk$ is algebraically 
closed, the assertion is a special case of Theorem \ref{xxthm5.8}(2).
If $\Bbbk$ is not algebraically closed, let $\Bbbk'$ be the closure of
$\Bbbk$. By Theorem \ref{xxthm5.8}(2), $\Bbbk' Q$ is strongly m-detectable
over $\Bbbk'$. By Lemma \ref{xxlem6.1}, $\Bbbk Q$ is strongly m-detectable
over $\Bbbk$, and then strongly m-cancellative by Lemmas \ref{xxex3.5}(2)
and \ref{xxlem3.6}(2).
\end{proof}

We need another lemma before proving the main result of this section. 
The ideas of the proof are similar to the proof of \cite[Lemma 4.6]{LWZ},
so the proof is omitted.

\begin{lemma}
\label{xxlem6.3}
Let $A$ and $B$ be two algebras.
\begin{enumerate}
\item[(1)]
If $A$ and $B$ are {\rm{(}}strongly{\rm{)}} m-cancellative, so is $A\oplus B$.
\item[(2)]
If $A$ and $B$ are {\rm{(}}strongly{\rm{)}} m-retractable, so is $A\oplus B$.
\item[(3)]
If $A$ and $B$ are {\rm{(}}strongly{\rm{)}} m-detectable, so is $A\oplus B$.
\end{enumerate}
\end{lemma}

Now we are ready to prove Theorem \ref{xxthm0.5}.

\begin{theorem}
\label{xxthm6.4} Let $Q$ be a finite quiver and let $A$ be the path algebra
$\Bbbk Q$. Then $A$ is strongly m-cancellative. If further $Q$ has no
connected component being $C_1$, then $A$ is strongly m-detectable.
\end{theorem}

\begin{proof} By Lemma \ref{xxlem6.3}, we may assume that $Q$ is
connected.

If $Q=C_1$, then $A=\Bbbk[x]$ and the assertion follows
Proposition \ref{xxpro5.7}(2).

If $Q=C_n$, then this is Lemma \ref{xxlem6.2}(2).

If $Q\neq C_n$ for any $n\geq 1$, then by \eqref{E6.0.1},
the center of $A$ is $\Bbbk$. By Theorem \ref{xxthm4.2}(1),
$A$ is strongly m-detectable. Since $A$ is ${\mathbb N}$-graded
and locally finite, it is strongly Hopfian by Example \ref{xxex3.5}(2).
By Theorem \ref{xxthm4.2}(2), $A$ is strongly m-cancellative. This
completes the proof.
\end{proof}

Theorem \ref{xxthm0.5} is clearly a consequence of the above
theorem.

\section{Comments, questions and examples}
\label{xxsec7}

One of the remaining questions in this project is to understand whether the 
cancellation property is equivalent to the m-cancellation property 
(as well as the d-cancellation property). We will make some comments 
about it in this section.

First of all we will show that three cancellation properties are equivalent
for commutative algebras. The next result was proved in \cite{YZ1} using 
slightly different wording.

\begin{proposition} \cite[Proposition 5.1]{YZ1}
\label{xxpro7.1}
Suppose that $A$ is an Azumaya algebra over its center $Z$ and 
that $\Spec Z$ is connected. If $D(A)$ and $D(B)$ are triangulated 
equivalent for another algebra $B$, then $A$ and $B$ are Morita 
equivalent.
\end{proposition}

Note that the Brauer group of a commutative
algebra $R$, denoted by $Br(R)$, is the set of
Morita-type-equivalence classes of Azumaya algebras over $R$,
in other words, $Br(R)$ classifies Azumaya algebras over $R$
up to an equivalence relation \cite{AG}. See \cite{Sc} for 
some discussion about the Brauer group. 
One immediate consequence is 

\begin{corollary}
\label{xxcor7.2}
Suppose $Z$ is a commutative algebra with $\Spec Z$ connected.
Then the following are equivalent.
\begin{enumerate}
\item[(1)]
$Z$ is {\rm{(}}strongly{\rm{)}} cancellative.
\item[(2)]
$Z$ is {\rm{(}}strongly{\rm{)}} m-cancellative.
\item[(3)]
$Z$ is {\rm{(}}strongly{\rm{)}} d-cancellative.
\end{enumerate}
\end{corollary}

\begin{proof} By Proposition \ref{xxpro7.1}, it remains to 
show that (1) and (2) are equivalent. By Lemma \ref{xxlem1.4},
part (1) follows from part (2). Now we show that part (2) is a 
consequence of part (1).

Suppose $A$ is an algebra such that $Z[\underline{t}]$ is 
Morita equivalent to $A[\underline{s}]$. By the map $\omega$
in \eqref{E2.1.1}, we obtain that $Z[\underline{t}]$ is
isomorphic to $Z(A)[\underline{s}]$. Since $Z$ is 
{\rm{(}}strongly{\rm{)}} cancellative, $Z(A)\cong Z$. Let us
identify $Z(A)$ with $Z$. Since $Z[\underline{t}]$ is 
Morita equivalent to $A[\underline{s}]$, $A[\underline{s}]$
is Morita equivalent to its center, which is $Z[\underline{s}]$. Then
the Brauer-class $[A[\underline{s}]]$ as an element in $Br(Z[\underline{s}])$
is trivial by \cite[Proposition 5.3]{AG}. Since the natural map
$Br(Z)\to Br(Z[\underline{s}])$ is injective, the Brauer-class 
$[A]$ as an element in $Br(Z)$ is trivial. By \cite[Theorem 4]{Sc}
or \cite[Proposition 4.1]{Ne}, $A$ is Morita equivalent to $Z$
as required.
\end{proof}

\begin{corollary}
\label{xxcor7.3} 
Let $Z$ be a {\rm{(}}strongly{\rm{)}} detectable commutative algebra 
such that $\Spec Z$ is connected.
If $A$ is an Azumaya algebra over $Z$ that is strongly 
Hopfian, then $A$ is both {\rm{(}}strongly{\rm{)}} 
m-cancellative and {\rm{(}}strongly{\rm{)}} d-cancellative.
\end{corollary}

\begin{proof} By Proposition \ref{xxpro7.1}, we  need
to show only the claim that $A$ is {\rm{(}}strongly{\rm{)}} m-cancellative. 
Since $A$ is strongly Hopfian, the claim follows from 
Lemmas \ref{xxlem3.6}(2) and \ref{xxlem3.7}(1).
\end{proof}

The next example is similar to \cite[Example 3.3]{LWZ}.

\begin{example}
\label{xxex7.4} Let $A=\Bbbk[x,y]/(x^2=y^2=xy=0)$. By Theorem \ref{xxthm4.1}, 
$A$ is strongly m-detectable. By \cite[Example 3.3]{LWZ} and Corollary 
\ref{xxcor7.2}, the commutative algebra $A$ is neither retractable nor 
m-retractable. 
\end{example}

For non-Azumaya (noncommutative) algebras, there is no general
approach to relate the m-cancellation property with the 
d-cancellation property. However, most of cancellative algebras
verified by using the discriminant method in \cite{BZ1} 
are m-cancellative as we will see next. 

Since most of algebras that we are interested in are strongly Hopfian, 
to show an algebra is m-cancellative, it suffices to show it is 
m-detectable [Lemma \ref{xxlem3.6}(1)]. By Lemma \ref{xxlem6.1},
under some mild hypotheses, we can assume the base field $\Bbbk$ is
algebraically closed. For simplicity, we assume that {\bf $\Bbbk$ is
algebraically closed of characteristic zero for the rest of this section}.

Let $I$ be an ideal of a commutative algebra $R$. Then the 
{\it radical} of $I$ is defined to be
$$\sqrt{I}=\bigcap_{{\mathfrak p}\in \Spec R, I\subseteq 
{\mathfrak p}} {\mathfrak p}.$$
The {\it standard trace} $tr_{st}$ defined in 
\cite[Sect. 2.1(2)]{BY} agrees with the {\it regular
trace} $tr_{reg}$  defined in \cite[p.758]{CPWZ2}.
So we take $tr=\tr_{st}=tr_{reg}$ in this paper.

\begin{proposition}
\label{xxpro7.5}
Let $A$ be a prime algebra that is finitely generated as a module over 
its center $Z$ and let $v$ be the rank of $A$ over $Z$.  
Let $D\subseteq Z$ be either the $v$-discriminant ideal $D_v(B: tr)$
in the sense of \cite[Definition 1.1(2)]{CPWZ2} or the modified 
$v$-discriminant ideal $MD_v(B: tr)$ in the sense of 
\cite[Definition 1.2(2)]{CPWZ2}. Suppose that
\begin{enumerate}
\item[(a)]
The center $Z$ is an affine domain and the standard trace $tr$ maps $A$ to $Z$.
\item[(b)]
$\sqrt{D}$ is a principal ideal of $Z$ generated by an element $f$.
\item[(c)]
$f$ is an effective {\rm{(}}respectively, dominating{\rm{)}} 
element in $Z$. 
\end{enumerate}
Then the following hold.
\begin{enumerate}
\item[(1)]
$A$ is strongly m-$Z$-retractable.
\item[(2)]
$A$ is strongly $Z$-retractable.
\item[(3)]
$A$ is strongly m-detectable.
\item[(4)]
$A$ is strongly m-cancellative.
\item[(5)]
$A$ is strongly cancellative.
\end{enumerate}
\end{proposition}

\begin{proof} Since we assume that $\Bbbk$ is algebraically closed 
of characteristic zero, we can apply \cite[Main Theorem]{BY} by taking
the standard trace. By \cite[Main Theorem]{BY}, we have
$${\mathcal V}(D)=\MaxSpec(Z)\setminus {\mathcal A}(A)$$
where ${\mathcal V}(D)$ is the zero-set of $D$. By Lemma \ref{xxlem5.6}(2),
${\mathcal A}(A)=L_{\mathcal S}(A)$ where ${\mathcal S}$ denotes
the property of being simple. Thus the ${\mathcal S}$-discriminant
set of $A$ is equal to ${\mathcal V}(D)$. As a consequence,
the ${\mathcal S}$-discriminant ideal of $A$ is equal to 
$I({\mathcal V}(D))$, which is $\sqrt{D}$. By hypothesis (b),
we obtain that the ${\mathcal S}$-discriminant ideal of $A$
is a principal ideal of $Z$ generated by an element $f$. Since
$f$ is effective (respectively, dominating), $Z$ is strongly
$\LND_{f}^H$-rigid by Theorem \ref{xxthm2.10}. 
Since ${\mathcal S}$ is a stable Morita invariant property
[Lemma \ref{xxlem5.1}], by Proposition \ref{xxpro2.7}(2), 
$A$ is both strongly m-$Z$-retractable and strongly $Z$-retractable.
Thus we proved parts (1) and (2). Note that part (3) follows from 
part (1) and Lemma \ref{xxlem3.4}. Since $A$ is noetherian,
it is strongly Hopfian [Example \ref{xxex3.5}(1)]. Parts (4)
and (5) follows from part (3) and Lemma \ref{xxlem3.6}(2).
\end{proof}

The next example is similar to \cite[Example 5.1]{LWZ}.

\begin{example}
\label{xxex7.6}
Let $R$ be an affine commutative domain and let $f$ be a product
of a set of generating elements of $R$. Let
$$A=\begin{pmatrix} R & fR \\R &R\end{pmatrix}.$$
It is easy to check that the (modified) $4$-discriminant of $A$ 
over its center $R$ is the ideal generated by $-f^2$. Clearly, 
the radical of $(-f^2)$ is the principal ideal $(f)$. By the above 
proposition, $A$ is strongly m-$Z$-retractable, m-detectable, 
m-cancellative and cancellative. 
\end{example}

Other precise examples are the following in which we omit 
some details. See also \cite[Example 4.8]{BZ1}.

\begin{example}
\label{xxex7.7} The following algebras are $m$-cancellative 
by verifying the hypotheses of Proposition \ref{xxpro7.5}.
\begin{enumerate}
\item[(1)]
Skew polynomial rings $\Bbbk_{q}[x_1,\cdots,x_{n}]$ when $n$
is an even number and $1\neq q$ is a root of unity. 
\item[(2)]
$\Bbbk\langle x,y\rangle/(x^2y-yx^2,y^2x+xy^2)$.
\item[(3)]
Quantum Weyl algebra $\Bbbk\langle x,y\rangle/(yx-qxy-1)$
where $1\neq q$ is a root of unity. 
\item[(4)]
Every finite tensor product of algebras of the form (1),(2) and (3).
\end{enumerate}
\end{example}

\subsection*{Acknowledgments}
The authors thank Daniel Rogalski for useful conversations on the 
subject and for providing the reference \cite{RSS}, thank Ellen
Kirkman for reading an earlier version of the paper and for her 
useful comments, and thank referees for their careful reading and
and valuable comments.
D.-M. Lu was partially supported by the National Natural
Science Foundation of China (Grant No. 11671351).
Q.-S. Wu was partially supported by the National Natural
Science Foundation of China
(Grant No. 11771085 and Key Project No. 11331006).
J.J. Zhang was partially supported by the US National Science Foundation
(Nos. DMS-1402863 and DMS-1700825).

\providecommand{\bysame}{\leavevmode\hbox to3em{\hrulefill}\thinspace}
\providecommand{\MR}{\relax\ifhmode\unskip\space\fi MR }
\providecommand{\MRhref}[2]{%

\href{http://www.ams.org/mathscinet-getitem?mr=#1}{#2} }
\providecommand{\href}[2]{#2}

\end{document}